\documentclass[12pt,reqno]{article}
mm \textheight=8.9in

\usepackage{amssymb}
\usepackage{amsmath}
\usepackage{amsthm}
\usepackage{pb-diagram}
\usepackage{color}

\newcommand{\br}[3]{{$#1$}$\lower4pt\hbox{$\tp\atop\raise4pt \hbox{$\scriptscriptstyle{#2}$}$} ${$#3$}}
\newcommand{\tw}[3]{{$#1$}${\,\scriptscriptstyle {#2}}\atop\raise9pt\hbox{$\scriptstyle\tp$} ${$#3$}}
\newcommand{\ttps}[2]{{#1}\raise5pt\hbox{$\lower12pt\hbox{$\scriptstyle\tp$}\atop \lower0pt\hbox{$\tilde\;$}$}\raise4.5pt\hbox{${\scriptstyle{#2}}$}}
\newcommand{\st}[1]{\mbox{${\,\scriptscriptstyle {#1}}\atop\raise5.5pt\hbox{$*$}$}}

\newcommand{\rd}[1]{\mbox{${\,\scriptscriptstyle {#1}}\atop\raise5.5pt\hbox{$\bullet$}$}}
\newcommand{\rt}[1]{\otimes_\chi}
\newcommand{\lt}[1]{\mbox{${\,\scriptscriptstyle {#1}}\atop\raise5.5pt\hbox{$\ltimes$}$}}
\newcommand{\btr}{\raise1.2pt\hbox{$\scriptstyle\blacktriangleright$}\hspace{2pt}}
\newcommand{\btl}{\raise1.2pt\hbox{$\scriptstyle\blacktriangleleft$}\hspace{2pt}}

\newcommand{\lcr}{\raise1.0pt \hbox{${\scriptstyle\rightharpoonup}$}}
\newcommand{\rcr}{\raise1.0pt \hbox{${\scriptstyle\leftharpoonup}$}}

\newcommand{\ttp}{{\lower12pt\hbox{$\tp$}\atop \hbox{$\tilde\;$}}}

\renewcommand{\P}{\mathrm{P}}
\newcommand{\id}{\mathrm{id}}

\newcommand{\sms}{\hspace{.7pt}}

\newcommand{\sgn}{\mathrm{sign}}

\newcommand{\Ru}{\mathcal{R}}

\newcommand{\Nc}{\mathcal{N}}

\newcommand{\C}{\mathbb{C}}
\newcommand{\Z}{\mathbb{Z}}

\newcommand{\N}{\mathbb{N}}

\newcommand{\tp}{\otimes}

\newcommand{\tht}{\theta}

\newcommand{\U}{U}

\newcommand{\ve}{\varepsilon}
\newcommand{\gm}{\gamma}
\newcommand{\dt}{\delta}

\newcommand{\op}{\oplus}
\newcommand{\la}{\lambda}

\newcommand{\End}{\mathrm{End}}

\newcommand{\Span}{\mathrm{Span}}

\newcommand{\rk}{\mathrm{rk}}

\newcommand{\Rm}{\mathrm{R}}

\newcommand{\g}{\mathfrak{g}}
\renewcommand{\b}{\mathfrak{b}}

\newcommand{\h}{\mathfrak{h}}

\newcommand{\lb}{\boldsymbol{l}}

\newcommand{\s}{\mathfrak{s}}

\renewcommand{\o}{\mathfrak{o}}

\newcommand{\nn}{\nonumber}
\newcommand{\p}{\mathfrak{p}}
\renewcommand{\l}{\mathfrak{l}}

\newcommand{\si}{\sigma}
\newcommand{\al}{\alpha}

\newcommand{\bt}{\beta}

\newcommand{\prt}{\partial}

\newcommand{\be}{\begin{eqnarray}}
\newcommand{\ee}{\end{eqnarray}}

\newtheorem{thm}{Theorem}[section]
\newtheorem{propn}[thm]{Proposition}
\newtheorem{lemma}[thm]{Lemma}

\newcount\prg

\newcommand{\parag}{\advance\prg by1 {\noindent\bf\thesection.\the\prg\hspace{6pt}}}

\begin{document}
\title{R-matrix and Mickelsson algebras for orthosymplectic quantum groups}
\author{
Thomas Ashton and Andrey Mudrov
\vspace{20pt}\\
\small Department of Mathematics,\\ \small University of Leicester, \\
\small University Road,
LE1 7RH Leicester, UK\\
\small e-mail: am405@le.ac.uk\\
}

\date{}
\maketitle

\begin{abstract}
Let $\g$ be a complex  orthogonal or symplectic Lie algebra and
 $\g'\subset \g$ the Lie subalgebra of rank $\rk \>\g'=\rk \> \g-1$ of the same type.
We give an explicit construction of  generators of the Mickelsson algebra $Z_q(\g,\g')$ in terms of Chevalley generators
via the R-matrix of $U_q(\g)$.
\end{abstract}


\newpage
\section{Introduction}
In the mathematics literature, lowering and rasing operators operators are known as generators of step algebras, which were originally
introduced by Mickelsson \cite{Mick} for reductive pairs of Lie algebras, $\g'\subset \g$. These algebras naturally act on  $\g'$-singular vectors in $U(\g)$-modules  and are important in representation theory,
\cite{Zh1,Mol}.

The general theory of step algebras for classical universal enveloping algebras was developed in \cite{Zh1,Zh2} and
was extended to the special liner and orthogonal quantum groups in \cite{Kek}. They admit a natural description in
terms of extremal projectors, \cite{Zh2}, introduced for classical groups in \cite{AST1,AST2}
and extended to the quantum group case in \cite{T}. It is known that the step algebra $Z(\g,\g')$ is generated by
the image of the orthogonal complement $\g\ominus \g'$ under the extremal projector of the $\g'$.
Another description of lowering/rasing operators for classical groups was obtained in
\cite{NM,Pei,PH,W} in an explicit form of polynomials in $\g$.

A generalization of the results of \cite{NM,Pei} to quantum $\g\l(n)$ can be found in \cite{ABHST}. In this special case, the lowering operators can be also conveniently expressed through  "modified commutators" in the Chevalley generators of $U(\g)$
with coefficients in the field of fractions of $U(\h)$. Extending
\cite{PH,W} to a general quantum group is not straightforward, since there are no
immediate candidates for the nilpotent triangular Lie subalgebras $\g_\pm $ in $U_q(\g)$. We suggest such a generalization, where the lack of
$\g_\pm $ is compensated by the entries of the universal R-matrix with one leg  projected to the natural representation.
Those entries are nicely expressed through  modified commutators in the Chevalley generators turning into elements of $\g_\pm$
in the quasi-classical limit. Their commutation relation with the Chevalley generators modify the classical commutation
relations with $\g_\pm$ in a tractable way.  This enabled us to generalize the results of \cite{NM,Pei,PH,W} and construct
generators of Mickelsson algebras for the non-exceptional quantum groups.

\subsection{Quantized universal enveloping algebra}
\label{ssecQUEA}
In this paper,  $\g$ is a complex simple Lie algebra of type $B$, $C$ or $D$.
The case of $\g\l(n)$ can be easily derived from here due to the natural inclusion $\U_q\bigl(\g\l(n)\bigr)\subset \U_q(\g)$, so we do not pay special attention to it.
We choose a Cartan subalgebra $\h\subset \g$ with the canonical inner product $(.,.)$ on $\h^*$.
By  $\Rm$ we denote the root system of $\g$ with a fixed subsystem of
positive roots $\Rm^+\subset \Rm$ and the basis of simple roots $\Pi^+\subset \Rm^+$.
For every $\la\in \h^*$ we denote by  $h_\la$  its image under the isomorphism $\h^*\simeq \h$,
that is $(\la,\bt)=\bt(h_\la)$ for all $\bt\in \h^*$.
We put $\rho=\frac{1}{2}\sum_{\al\in \Rm^+}\al $ for the Weyl vector.

Suppose that $q\in \C$ is not a root of unity. Denote by $U_q(\g_\pm)$ the $\C$-algebra generated by  $e_{\pm\al}$, $\al\in \Pi^+$, subject to the q-Serre relations
$$
\sum_{k=0}^{1-a_{ij}}(-1)^k
\left[
\begin{array}{cc}
1-a_{ij} \\
 k
\end{array}
\right]_{q_{\al_i}}
e_{\pm \al_i}^{1-a_{ij}-k}
e_{\pm \al_j}e_{\pm \al_i}^{k}
=0
,
$$
where  $a_{ij}=\frac{2(\al_i,\al_j)}{(\al_i,\al_i)}$,
$i,j=1,\ldots, n=\rk\: \g$, is the Cartan matrix, $q_{\al}= q^{\frac{(\al,\al)}{2}}$, and
$$
\left[
\begin{array}{cc}
m  \\ k
\end{array}
\right]_{q}
=
\frac{[m]_q!}{[k]_q![m-k]_q!},
\quad
[m]_q!=[1]_q\cdot [2]_q\ldots [m]_q.
$$
Here and further on, $[z]_q=\frac{q^z-q^{-z}}{q-q^{-1}}$ whenever $q^{\pm z}$ make sense.

Denote by $U_q(\h)$ the commutative $\C$-algebra generated by $q^{\pm h_\al}$, $\al\in \Pi^+$. The quantum group $U_q(\g)$ is a $\C$-algebra generated by  $U_q(\g_\pm)$ and $U_q(\h)$ subject
to the relations
$$
q^{ h_\al}e_{\pm \bt}q^{-h_\al}=q^{\pm(\al,\bt)} e_{\pm \bt},
\quad
[e_{\al},e_{-\bt}]=\delta_{\al, \bt}\frac{q^{h_{\al}}-q^{-h_{\al}}}{q_\al-q^{-1}_\al}.
$$
Remark that $\h$ is not contained in $U_q(\g)$, still it is convenient for us to keep reference to $\h$.

Fix the comultiplication in $U_q(\g)$ as in \cite{CP}:
\be
&\Delta(e_{\al})=e_{\al}\tp q^{h_{\al}} + 1\tp e_{\al},
\quad
\Delta(e_{-\al})=e_{-\al}\tp 1 + q^{-h_{\al}} \tp e_{-\al},
\nn\\&
\Delta(q^{\pm h_{\al}})=q^{\pm h_{\al}}\tp q^{\pm h_{\al}},
\nn
\ee
for all $\al \in \Pi^+$.

The subalgebras $U_q(\b_\pm)\subset U_q(\g)$ generated by $U_q(\g_\pm)$ over $U_q(\h)$ are quantized universal enveloping algebras of the
Borel subalgebras $\b_\pm=\h+\g_\pm\subset \g$.

The Chevalley  generators $e_{\al}$ can be extended to a set of higher root vectors $e_{\bt}$ for
all $\bt\in \Rm$. A normally ordered set of root vectors generate a Poincar\'{e}-Birkhoff-Witt (PBW) basis of $U_q(\g)$
over $U_q(\h)$, \cite{CP}. We will use $\g_\pm$ to denote the vector space spanned by $\{e_{\pm \bt}\}_{\bt\in \Rm^+}$.

The universal R-matrix is an element of a certain extension of $U_q(\g)\tp U_q(\g)$.
We heavily use the intertwining relation
\be
\Ru \Delta(x)= \Delta^{op}(x)\Ru,
\label{intertiner}
\ee
between the coproduct and its opposite for all $x\in U_q(\g)$.
Let $\{\ve_i\}_{i=1}^n\subset \h^*$ be the standard orthonormal basis  and $\{h_{\ve_i}\}_{i=1}^n$ the corresponding dual basis in $\h$.
The exact expression for $\Ru$ can be extracted from \cite{CP}, Theorem 8.3.9, as the ordered product
\be
\Ru= q^{\sum_{i=1}^n h_{\ve_i}\tp  h_{\ve_i}}\prod_{\bt} \exp_{q_\bt}\{(1-q_\bt^{-2})(e_\bt\tp e_{-\bt} )\} \in U_q(\b_+)\hat \tp U_q(\b_-),
\label{Rmat}
\ee
where  $\exp_{q}(x )=\sum_{k=0}^\infty q^{\frac{1}{2}k(k+1)}\frac{x^k}{[k]_q!}$.

We use the notation $e_i=e_{\al_i}$ and $f_{i}=e_{-\al_i}$ for  $\al_i\in \Pi^+$, in all cases apart
from $i=n$, $\g=\s\o(2n+1)$, where we set $f_n=[\frac{1}{2}]_q e_{-\al_n}$.
The reason for this is two-fold.
Firstly, the natural representation can be defined through the classical assignment on the generators, as given below.
Secondly,
we get rid of $q_{\al_n}=q^{\frac{1}{2}}$ and can work over $\C[q]$, as the relations involved turn into
$$
[e_{n},f_{n}]=\frac{q^{h_{\al_n}}-q^{-h_{\al_n}}}{q-q^{-1}},
$$
$$
f_{n}^3f_{n-1}-(q+1+q^{-1})f_{n}^2f_{n-1}f_{n}+(q+1+q^{-1})f_{n}f_{n-1}f_{n}^2-f_{n-1}f_{n}^3=0.
$$
It is easy to see that the square root of $q$  disappears from the corresponding factor in the presentation
(\ref{Rmat}).

In what
follows, we regard $\g\l(n)\subset \g$ to be the Lie subalgebra with the simple roots $\{\al_i\}_{i=1}^{n-1}$ and $U_q\bigl(\g\l(n)\bigr)$ the corresponding
quantum subgroup in $U_q(\g)$.

Consider the natural representation of $\g$ in the vector space $\C^N$.
We use the notation $i'=N+1-i$ for all integers $i=1,\ldots,N$. The assignment
$$
\pi(e_{i})=e_{i,i+1}\pm e_{i'-1,i'}, \quad \pi(f_{i})= e_{i+1,i}\pm e_{i',i'-1}, \quad \pi(h_{\al_i})= e_{ii}-e_{i+1,i+1}+e_{i'-1,i'-1}-e_{i'i'},
$$
for $i=1, \ldots,n-1$, defines a direct sum of two representations of $\g\l(n)$ for each sign.
It extends to the natural representation of the whole $\g$ by
$$
\pi(e_{n})= e_{n,n+1}\pm e_{n'-1,n'}, \quad \pi(f_{n})=  e_{n+1,n}\pm e_{n',n'-1}, \quad \pi(h_{\al_n})=
e_{nn}-e_{n'n'},
$$
$$
\pi(e_{n})= e_{nn'}, \quad \pi(f_{n})= e_{n'n}, \quad \pi(h_{\al_n})=
2e_{nn}-2e_{n'n'},
$$
$$
\pi(e_{n})= e_{n-1,n'}\pm e_{n,n'+1}, \quad \pi(f_{n})= e_{n',n-1}\pm e_{n'+1,n}, \quad \pi(h_{\al_n})=
e_{n-1,n-1}+e_{nn}-e_{n'n'}-e_{n'+1,n'+1},
$$
respectively, for $\g=\s\o(2n+1)$, $\g=\s\p(2n)$, and $\g=\s\o(2n)$.

Two values of the sign give equivalent representations.
The choice of minus corresponds to the standard representation that preserves the bilinear form
with entries $C_{ij}=\dt_{i'j}$, for $\g=\s\o(N)$, and $C_{ij}=\sgn(i'-i)\dt_{i'j}$, for $\g=\s\p(N)$.
However, we fix the sign to $+$ in order to simplify calculations. The above assignment also defines representations of $U_q(\g)$.

\section{$R$-matrix of non-exceptional quantum groups}
Define $\check{\Ru}=q^{-\sum_{i=1}^n h_{\ve_i}\tp h_{\ve_i}}\Ru $.
Denote by $\check{R}^-=(\pi\tp \id)(\check{\Ru}) \in \End(\C^N)\tp U_q(\g_-)$
and by $\check{R}^+=(\pi\tp \id)(\check{\Ru}_{21}) \in \End(\C^N)\tp U_q(\g_+)$.
In this section, we deal only with $\check{R}^-$ and suppress the label "$-$" for simplicity,
$\check{R}=\check{R}^-$.

Denote by $N_+$ the ring of all upper triangular matrices in $\End(\C^N)$ and by $N'_+$ its ideal spanned by $e_{ij}$, $i<j+1$.
\begin{lemma}
One has
$$
\check{R} =1\tp 1+(q^{1+\dt_{1n}}-q^{-1-\dt_{1n}})\sum_{i=1}^n\pi(e_i)\tp f_{i}\mod N'_+\tp U_q(\g_-),
$$
where $\dt_{1n}$ is present only for $\g=\s\p(2n)$.
\end{lemma}
\begin{proof}
For all positive roots $\al, \bt$ the matrix $\pi(e_\al e_\bt)$ belongs to $N'_+$.
Also, $\pi(e_\bt)\in N'_+$ for all $\bt\in \Rm^+\backslash \Pi^+$.
 Therefore, the only terms that
contribute to $\Span_{\ve_i-\ve_j\in \Pi^+}\{e_{ij}\tp U_q(\g_-)\}$
are those of degree 1 from the  series $\exp_{q_\al}(1-q_\al^{-2})(e_\al\tp e_{-\al} )$ with $\al\in \Pi^+$.
\end{proof}
\noindent
Write $\check{R}=\sum_{i,j=1}^N e_{ij}\tp \check{R}_{ij}$, where $\check{R}_{ij}=0$ for $i>j$. Due to the $\h$-invariance
of $\check R$, the entry $\check{R}_{ij}\in U_q(\g_-)$ carries weight $\ve_j-\ve_i$.

For all $\g$, we have $f_{k,k+1}=f_k=f_{k'-1,k'}$ once $k<n$
and $f_{n,n+1}=f_n=f_{n+1, n'}$ for $\g=\s\o(2n+1)$,
$f_{n-1,n'}=f_n=f_{n, n'+1}$ for $\g=\s\o(2n)$, and
$f_{nn'}=[2]_qf_n$ for $\g=\s\p(2n)$.
We present explicit expressions for the  entries $f_{ij}$ in terms of modified commutators in Chevalley generators,
 $[x,y]_a=xy- ayx$, where $a$ is a  scalar; we also put $\bar q = q^{-1}$.
\begin{propn}
\label{expl_f}
Suppose that $\ve_i-\ve_j \in \Rm^+\backslash \Pi^+$. Then the elements $f_{ij}$ are given by the following formulas:\\
For all $\g$ and  $i+1<j\leqslant\frac{N+1}{2}$:
\be
f_{ij}=[f_{j-1},\ldots [f_{i+1},f_i]_{\bar q^{}}\ldots ]_{\bar q},
\quad
f_{j'i'}=[\ldots [f_i,f_{i+1}]_{\bar q^{}},\ldots f_{j-1}]_{\bar q}.
\label{gl(n)}
\ee
Furthermore,
\begin{itemize}
\item
For $\g=\s\o(2n+1)$: $f_{nn'}=(q-1)f_n^2$ and
$$
f_{i,n+1}=[f_{n},f_{i,n}]_{\bar q^{}}, \quad f_{n+1,i'}=[f_{n',i'},f_{n}]_{\bar q^{}}, \quad i<n,
$$
$$f_{ij'}=q^{\dt_{ij}}[f_{n+1,j'},f_{i,n+1}]_{\bar q^{\dt_{ij}}},  \quad i,j<n.$$
\item
For $\g=\s\p(2n)$:  $f_{nn'}=[2]_qf_n$ and
$$
f_{in'}=[f_{n},f_{in}]_{\bar q^{2}}, \quad f_{ni'}=[f_{n'i'},f_{n}]_{\bar q^{2}}, \quad i<n,
$$
$$f_{ij'}=q^{\dt_{ij}}[f_{nj'},f_{in}]_{\bar q^{1+\dt_{ij}}}, \quad i,j<n.
$$
\item
For $\g=\s\o(2n)$: $f_{nn'}=0$ and
$$f_{in'}=[f_{n},f_{i,n-1}]_{\bar q^{}}, \quad f_{ni'}=[f_{n'+1,i'},f_n]_{\bar q^{}},\quad i<n-2,$$
$$
f_{ji'}=q^{\dt_{ij}}[f_{ni'},f_{j,n}]_{\bar q^{1+\dt_{ij}}}, \quad i,j\leqslant n-1.
$$
\end{itemize}
\end{propn}
\begin{proof}
The proof is a direct calculation with the use of the identity
$$
(f_\al\tp 1)\check{\Ru}- \check{\Ru}(f_\al\tp 1)  =\check{\Ru}(q^{-h_\al}\tp f_\al)-(q^{h_\al}\tp f_\al)\check{\Ru},
$$
which follows from
the intertwining axiom (\ref{intertiner}) for $x=f_\al$.
This allows us to construct the elements $f_{ij}$ by induction  starting from $f_\al$, $\al\in \Pi^+$.
\end{proof}
For each $\al\in \Pi^+$, denote by $P(\al)$ the set of ordered pairs $l,r=1,\ldots,N$,
with $\ve_l-\ve_r= \al$. We call such pairs simple.
\begin{propn}
\label{non-dyn_f_ij}
The matrix entries $f_{i,j}\in U_q(\g_-)$ such that  $\ve_i-\ve_j \not \in \Pi^+$
satisfy the identity
$$[e_{\al},f_{ij}]=\sum_{(l,r)\in P(\al)}
\bigl(f_{il}\dt_{jr}q^{h_{\al}}-q^{-h_{\al}}\dt_{il}f_{rj}\bigr),$$
for all simple positive  roots  $\al$.
\end{propn}
\begin{proof}
The proof is a straightforward calculation based on the intertwining relation  (\ref{intertiner}), which
is equivalent to
$$
(1\tp e_\al) \check{\Ru}-\check{\Ru}(1\tp e_\al)=\check{\Ru}(e_\al\tp q^{h_\al})-(e_\al\tp q^{-h_\al}) \check{\Ru},
$$
for $x=e_\al$, $\al\in \Pi^+$. Alternatively, one can use the  expressions for $f_{ij}$ from Proposition \ref{expl_f}.
\end{proof}
\section{Mickelsson algebras}
Consider the Lie subalgebra $\g'\subset \g$ corresponding to the root subsystem $\Rm_{\g'}\subset \Rm_\g$ generated by $\al_i$, $i>1$, and
let $\h'\subset \g'$ denote its Cartan subalgebra.
Let the triangular decomposition $\g'_-\op \h'\op \g'_+$ be compatible with the triangular decomposition of $\g$. Recall the definition of step algebra
$Z_q(\g,\g')$ of the pair $(\g,\g')$.
Consider the left ideal $J=U_q(\g)\g'_+$ and its normalizer  $\Nc=\{x\in U_q(\g): e_\al x \subset J, \forall \al \in \Pi^+_{\g'}\}$. By construction, $J$ is a two-sided ideal in the algebra $\Nc$. Then $Z_q(\g,\g')$ is
the quotient $\Nc/J$.

For all $\bt_i\in \Rm^+_\g\backslash \Rm^+_{\g'}$ let $e_{\bt_i}$ be the corresponding PBW generators
and let  $Z$ be  the vector space spanned by $e_{-\bt_l}^{k_l}\ldots e_{-\bt_1}^{k_1}e_0^{k_0} e_{\bt_1}^{m_1}\ldots e_{\bt_l}^{m_l}$,
were  $e_0=q^{h_{\al_1}}$, $k_i\in \Z_+$, and $k_0\in \Z$.
The PBW factorization
$
U_q(\g)=U_q(\g'_-)Z  U_q(\h') U_q(\g'_+)
$
gives rise to the decomposition
$$
U_q(\g)=Z U_q(\h') \op (\g'_- U_q(\g)+ U_q(\g)\g'_+).
$$
\begin{propn}[\cite{Kek}, Theorem 1]
The projection $U_q(\g)\to Z U_q(\h')$ implements an embedding of $Z_q(\g,\g')$ in $Z U_q(\h')$.
\label{Kekalainen}
\end{propn}
\begin{proof}
The  statement is proved in  \cite{Kek} for the orthogonal and special linear quantum groups
 but the arguments apply to symplectic groups too.
\end{proof}
It is proved within the theory of extremal projectors that
generators 
of $Z_q(\g,\g')$ are labeled by
the roots $\bt\in \Rm_\g\backslash \Rm_{\g'}$ plus $z_0=q^{h_{\al_1}}$.
We calculate them in the subsequent sections, cf. Propositions \ref{negative} and \ref{positive}.

\subsection{Lowering operators}
In what follows, we extend $U_q(\g)$ along with its subalgebras containing $U_q(\h)$ over the field of fractions of $U_q (\h)$
and denote such an extension by hat, e.g. $\hat U_q(\g)$. In this section we calculate representatives
of the negative generators of $Z_q(\g,\g')$  in $\hat U_q(\b_-)$.

Set $h_i=h_{\ve_i}\in \h$ for all $i=1,\ldots,N$ and introduce $\eta_{ij}\in \h+\C$ for $i,j=1,\ldots,N$, by
\be
\eta_{ij}=h_i-h_j+(\ve_i-\ve_j, \rho)-\frac{1}{2}|\!|\ve_i-\ve_j|\!|^2.
\ee
Here $|\!|\mu|\!|$ is the Euclidean norm on $\h^*$.
\begin{lemma}
\label{thetas}
Suppose that $(l,r)\in P(\al)$ for some $\al\in \Pi^+$. Then
\begin{itemize}
\item[i)] if  $l<r<j$, then $\eta_{lj}-\eta_{rj}=h_\al+(\al,\ve_j-\ve_r)$,
\item[ii)] if  $i<l<r $, then $\eta_{li}-\eta_{ri}=h_\al+(\al,\ve_i-\ve_r)$,
\item[iii)]
 $\eta_{lr}=h_\al.$
\end{itemize}
\end{lemma}
\begin{proof}
We have $(\al, \rho)=\frac{1}{2}|\!|\al|\!|^2$ for all $\al \in \Pi^+$. This proves iii). Further, for $\ve_l-\ve_r=\al$:
\be
\eta_{lj}-\eta_{rj}&=&h_\al+\frac{1}{2}|\!|\al|\!|^2+\frac{1}{2}|\!|\ve_j-\ve_r|\!|^2-\frac{1}{2}|\!|\ve_j-\ve_r-\al|\!|^2=h_\al+(\al,\ve_j-\ve_r),
\quad r< j,
\nn
\\
\eta_{li}-\eta_{ri}&=&h_\al+\frac{1}{2}|\!|\al|\!|^2+\frac{1}{2}|\!|\ve_i-\ve_r|\!|^2-\frac{1}{2}|\!|\ve_i-\ve_r-\al|\!|^2=h_\al+(\al,\ve_i-\ve_r),
\quad i< l,
\nn
\ee
which proves i) and ii).
\end{proof}

We call a strictly ascending sequence $\vec m=(m_1,\ldots,m_s)$ of integers a route from $m_1$ to $m_s$.
We write $m<\vec m$ and $\vec m< m$ for $m\in \Z$ if, respectively, $m< \min\vec m$ and $\max\vec m< m$. More generally,
we write $\vec m< \vec k$ if $\max\vec m < \min \vec k$.
In this case, a sequence $(\vec m, \vec k)$ is a route from $\min \vec m$ to $\max\vec k$.

Given a route $\vec m=(m_1,\ldots,m_s)$, define the product
$
f_{\vec m}=f_{m_1,m_2} \cdots f_{m_{s-1},m_s}\in U_q(\g_-)$.
Consider a free right $\hat U_q(\h)$-module $\Phi_{1m}$ generated by $f_{\vec m}$ with $1\leqslant\vec m \leqslant  j$ and
define an operation $\partial_{lr}\colon \Phi_{1j}\to \hat U_q(\b_-)$ for $(l,r)\in P(\al)$ as follows. Assuming
$1\leqslant \vec\ell < l< r< \vec \rho < j$, set
$$
\begin{array}{rrccc}
\partial_{lr}f_{(\vec\ell,l)}f_{(l,r)}f_{(r,\vec \rho)}&=&f_{(\vec\ell,l)}f_{(r,\vec \rho)}[\eta_{lj}-\eta_{rj}]_q,\\
\partial_{lr}f_{(\vec\ell,l)}f_{(l,\vec \rho)}&=&-f_{(\vec\ell,l)}f_{(r,\vec \rho)}q^{-\eta_{lj}+\eta_{rj}},
\\
\partial_{lr}f_{(\vec\ell,r)}f_{(r,\vec \rho)}&=&f_{(\vec\ell,l)}f_{(r,\vec \rho)}q^{\eta_{lj}-\eta_{rj}},
\\
\partial_{lr}f_{\vec m}&=&0, & l\not\in \vec m, &r \not \in \vec m.\\
\end{array}
$$
Extend $\partial_{lr}$ to entire $\Phi_{1j}$  by $\hat U_q(\h)$-linearity. Let $p\colon \Phi_{1j}\to \hat U(\g)$
denote the natural homomorphism of $\hat U_q(\h)$-modules.
\begin{lemma}
\label{partial}
For all $\al \in \Pi^+$   and all $x\in \Phi_{1j}$,
$e_{\al}\circ p(x)=\sum_{(l,r)\in P(\al)}\partial_{lr}x \mod \hat U_q(\g)e_\al$.
\end{lemma}
\begin{proof}
A straightforward analysis based on Proposition \ref{non-dyn_f_ij} and Lemma \ref{thetas}.
\end{proof}
\noindent
To simplify the presentation, we suppress the symbol of projection $p$ in what follows.

Introduce elements
$A_{r}^{j}\in  \hat U_q(\h)$ by
\be
A_{r}^{j}&=&\frac{q-q^{-1}}{q^{-2\eta_{rj}}-1},
\label{A_r}
\ee
for all $r,j\in [1,N]$ subject to $r< j$.
For each simple pair $(l,r)$ we define  $(l,r)$-chains as
\be
f_{(\vec \ell,l)}f_{(l,\vec\rho)}A_{l}^{j}+f_{(\vec\ell,l)}f_{(l,r)}f_{(r,\vec\rho)}A_{l}^{j}A_{r}^{j}+f_{(\vec\ell,r)}f_{(r,\vec\rho)}A_{r}^{j},
\quad f_{(\vec\ell,l)}f_{l,j}A_{l}^{j}+f_{(\vec\ell,j)},
\label{chain}
\ee
where $1\leqslant\vec \ell < l$ and $r<\vec\rho \leqslant j$. Remark that $f_{(l,r)}=\left[\frac{(\al,\al)}{2}\right]_q e_{-\al}$, where $\al=\ve_l-\ve_r$.
\begin{lemma}
\label{open-right}
The operator $\prt_{lr}$ annihilates $(l,r)$-chains.
\end{lemma}
\begin{proof}
Applying $\partial_{lr}$ to the 3-chain in (\ref{chain}), we get
$$
f_{(\vec\ell,l)}f_{(r,\vec\rho)}(-q^{-\eta_{lj}+\eta_{rj}}A_{l}^{j}+[\eta_{lj}-\eta_{rj}]_qA_{l}^{j}A_{r}^{j}+q^{\eta_{lj}-\eta_{rj}}A_{r}^{j}).
$$
The factor in the brackets turns zero on substitution of \ref{A_r}.

Now apply $\partial_{lj}$ to the right expression in (\ref{chain}) and get
$$
f_{(\vec\ell,l)}([h_\al]_qA_{l}^{j}+q^{h_\al})
=f_{(\vec\ell,l)}(\frac{q^{h_\al}-q^{-h_\al}}{q^{-2\eta_{lj}}-1}+q^{h_\al})
=f_{(\vec\ell,l)}\frac{[h_\al-\eta_{lj}]_q}{[-\eta_{lj}]_q}=0,
$$
so long as $\eta_{lj}=h_\al$ by Lemma \ref{thetas}.
\end{proof}

Given a route $\vec m=(m_1,\ldots,m_s)$, put
 $A_{\vec m}^{j}= A_{m_1}^{j}\cdots  A_{m_s}^{j}\in \hat U_q(\h)$ (and $A_{\vec m}^{j}=1$ for the empty route) and define
\be
z_{-j+1}=\sum_{1< \vec m < j}f_{(1,\vec m,j)} A_{\vec m}^{j} \in \hat U_q(\b_-),\quad j=2,\ldots,N,
\label{dyn_f}
\ee
 where the summation is taken over all possible $\vec m$
subject to the specified inequalities  plus the empty route.

\begin{propn}
\label{negative}
$
e_{\al} z_{-j}=
0 \mod
\hat U_q(\g)e_\al
$
for all $\al \in \Pi^+_{\g'}$ and $j=1,\ldots,N-1$.
\end{propn}
\begin{proof}
Thanks to Lemma \ref{partial}, we can reduce consideration to the action of operators $\prt_{lr}$,
with $(l,r)\in P(\al)$. According to the definition of $\prt_{lr}$
the summands in (\ref{dyn_f}) that survive the action of $\prt_{lr}$ can be organized into
a linear combination of $(l,r)$-chains with coefficients in $\hat U_q(\h)$. By Lemma \ref{open-right} they are killed by $\prt_{lr}$.
\end{proof}
The elements  $z_{-i}$, $i=1,\ldots,N-1$,
belong to the normalizer $\Nc$ and form the set of negative generators of $Z_q(\g,\g')$
for symplectic $\g$. In the orthogonal case, the negative part of $Z_q(\g,\g')$ is generated by $z_{-i}$, $i=1,\ldots,N-2$.
\subsection{Raising operators}
In this section we construct positive generators of $Z_q(\g,\g')$, which are called raising
operators.
Consider an algebra automorphism $\omega\colon U_q(\g)\to U_q(\g)$ defined on the generators
by $f_\al\leftrightarrow e_\al$, $q^{\pm h_\al}\mapsto q^{\mp h_\al}$. For  $i<j$,  let $g_{ji}$ be the image of $f_{ij}$ under this isomorphism. The natural representation  restricted to $U_q(\g_\pm)$ intertwines $\omega$ and matrix transposition.
Since $(\omega\tp \omega)(\check{\Ru})=\check\Ru_{21}$, the matrix  $\check{R}^+=(\pi\tp \id)(\check\Ru_{21})$ is equal to $1\tp 1+(q-q^{-1})\sum_{i<j}e_{ji}\tp g_{ji}$.
\begin{lemma}
For all $\al\in \Pi^+_{\g'}$ and all $i>1$,
$
e_\al g_{i1}=\sum_{(l,r)\in P(\al)}\dt_{il} g_{r1} \mod \hat  U_q(\g)e_\al
$.
\label{[e,g]}
\end{lemma}
\begin{proof}
Follows from the intertwining property of the R-matrix.
\end{proof}
Consider the right $\hat U_q(\h)$-module $\Psi_{i1}$ freely generated by $f_{(\vec m,k)}g_{k1}$ with $i\leqslant\vec m< k$.
We define  operators $\prt_{lr}\colon \Psi_{i1}\to \hat U_q(\g)$ similarly as we did it for $\Phi_{1j}$.
For a simple pair $(l,r)\in P(\al)$, put
$$
\partial_{l,r}f_{(\vec m,k)}g_{k1}=
\left\{
\begin{array}{rrrrr}
f_{(\vec m,l)}g_{r1},&l =k,\\
\bigl(\prt_{l,r}f_{(\vec m,k)}\bigr)g_{k1},&l\not =k,
\end{array}
\right.
 \quad i\leqslant \vec m< r.
$$
The Cartan factors appearing in $\prt_{lr}f_{(\vec m,k)}$ depend on $h_\al$. When pushed to the right-most position,
$h_\al$ is shifted by $(\al,\ve_1-\ve_r)$.
We extend $\prt_{lr}$ to an action on $\Psi_{i1}$ by the requirement that $\prt_{lr}$ commutes with the right action of $\hat U_q(\h)$. Let $p$ denote the natural homomorphism of $\hat U_q(\h)$-modules,
$p\colon \Psi_{i1}\to \hat U_q(\g)$.
One can prove the following analog of Lemma \ref{partial}.
\begin{lemma}
\label{partial_+}
For all $\al \in \Pi^+_{\g'}$  and all $x \in \Psi_{i1}$,
$e_{\al}\circ p(x)=\sum_{(l,r)\in P(\al)} \prt_{lr} x \mod \hat U_q(\g)e_\al$.
\end{lemma}
\begin{proof}
Straightforward.
\end{proof}
\noindent
We suppress the symbol of projection $p$ to simplify the formulas.

Define $\si_i$ for all $i=1,\ldots,N$ as follows. For $i<j$ let $|\!\!|i-j|\!\!|$ be the number of simple positive roots
entering $\ve_i-\ve_j$. For all $i,k= 2,\ldots, N$, $i< k$, put
$$
A_k^i=\frac{q^{\eta_{k1}-\eta_{i1}}}{[\eta_{i1}-\eta_{k1}]_q}, \quad B_k^i=\frac{(-1)^{|\!\!|i-k|\!\!|}}{[\eta_{i1}-\eta_{k1}]_q},
$$
For each $(l,r)\in \P(\al)$, where $\al \in \Pi^+_{\g'}$, define 3-chains  as
\be
f_{(i,\vec m, l)}g_{l1}B_{l}^i
+
f_{(i,\vec m, l)}f_{(l,r)}g_{r1}A_{l}^iB_{r}^i
+
f_{(i,\vec m, r)}g_{r1}B_{r}^i,
\label{3-chain_AB}
\ee
with $i< \vec m< l< r \leqslant N$ and
\be
f_{(i,\vec \ell, l)}f_{(l,\vec \rho, k)}g_{k1}A_{l}^i
+
f_{(i,\vec \ell, l)}f_{(l,r)}f_{(r,\vec \rho, k)}g_{k1}A_{l}^iA_{r}^i
+
f_{(i,\vec \ell, r)}f_{(r,\vec \rho, k)}g_{k1}A_{r}^i
\label{3-chain_AA}
\ee
with
$i< \vec \ell <  l  < r < \vec \rho <  k\leqslant N$.
The 2-chains are defined as
\be
g_{i1}+f_{(i,r)}g_{r1}B_{r}^i, \quad
 f_{(i,\vec m, k)}g_{k1}+f_{(i,r)}f_{(r,\vec m, k)}g_{k1}A_{r}^i
 \label{2-chain AB}
\ee
where $r$ is such that $\ve_i-\ve_r\in \Pi^+_{\g'}$
and $i< r< \vec m< k \leqslant N$. In all cases the empty routes $\vec m$ are admissible.
\begin{lemma}
\label{chains_killed}
For all $\al \in \Pi^+_{\g'}$ and all $(l,r)\in P(\al)$ the $(l,r)$-chains are
annihilated by $\prt_{lr}$.
\end{lemma}
\begin{proof}
Suppose that $i=l$ and apply $\prt_{ir}$ to the left 2-chain in (\ref{2-chain AB}). The result is
$$
g_{r1}+[h_\al]_qg_{r1}B^{i}_r
=g_{r1}(1+[h_{\al} +(\al,\ve_1-\ve_r)]_qB_{r}^i)=g_{r1}(1+[\eta_{i1}-\eta_{r1}]_qB_{r}^i)=0,
$$
by Lemma \ref{thetas}.
Applying $\prt_{ir}$ to the  right  2-chain  in (\ref{2-chain AB}) we get
$$
f_{(r,\vec m, k)}g_{k1}(-q^{-\eta_{i1}+\eta_{r1}}+[\eta_{i1}-\eta_{r1}]_qA_{r}^i)=0.
$$
Now consider 3-chains. The  action of $\prt_{lr}$ on the (\ref{3-chain_AA}) produces
\be
-f_{(i,\vec \ell, l)}q^{-h_\al}f_{(r,\vec \rho, k)}g_{k,1}A_{l}^i
+
f_{(i,\vec \ell, l)}[h_\al]_qf_{(r,\vec \rho, k)}g_{k,1}A_{l}^iA_{r}^i
+
f_{(i,\vec \ell, l)}q^{h_\al}f_{(r,\vec \rho, k)}g_{k,1}A_{r}^i,
\nn
\ee
which turns zero since
$
-q^{\eta_{r1}-\eta_{l1}}A_{l}^i
+
[\eta_{l1}-\eta_{r1}]_qA_{l}^iA_{r}^i
+
q^{\eta_{l1}-\eta_{r1}}A_{r}^i=0
$.
The action of $\prt_{lr}$ on (\ref{3-chain_AB}) yields
$$
f_{(i,\vec m, l)}g_{r1}B_{l}^i
+
f_{(i,\vec m, l)}[h_\al]g_{r1}A_{l}^iB_{r}^i
+
f_{(i,\vec m, l)}q^{h_\al}g_{r1}B_{r}^i.
$$
This is vanishing since
$
B_{l}^i
+
[\eta_{l1}-\eta_{r1}]A_{l}^iB_{r}^i
+
q^{\eta_{l1}-\eta_{r1}}B_{r}^i
=
B_{l}^i
+
\frac{[\eta_{i1}-\eta_{r1}]_q}{[\eta_{i1}-\eta_{l1}]_q}B_{r}^i=0.
$
\end{proof}
Given a route $\vec m = (m_1,\ldots, m_k)$ such that $i<\vec m$ let $A_{\vec m}^i$ denote the product $A^i_{m_1}\ldots A^i_{m_k}$.
Introduce elements $z_{i}\in \hat U_q(\g_-)\g_+$ of weight $\ve_1-\ve_i$ by
$$
z_{i-1}=g_{i1}+\sum_{i< \vec m< k\leqslant N}f_{(i,\vec m,k)}g_{k1}A_{\vec m}^iB_k^i, \quad i=2,\ldots,N.
$$
Again, the summation includes empty $\vec m$.
\begin{propn}
$
e_{\al} z_{i}=
0 \mod
\hat U_q(\g)e_\al
$,
for all $\al \in \Pi^+_{\g'}$
and $i=1,\ldots,N-1$.
\label{positive}
\end{propn}
\begin{proof}
By Lemma \ref{[e,g]},  the vectors $g_{2'1}$ and $z_{N-1}=g_{1'1}$ are normalizing the left ideal $\hat U_q(\g)\g'_+$,
so is $z_{N-2}=g_{2'1}+f_1g_{1'1}B^{1'}_{2'}$.
Once the cases $i=2',1'$ are proved, we  further assume $i<2'$.
In view of Lemma \ref{partial_+}, it is sufficient to show that $z_{i-1}$ is killed, modulo $\hat U_q(\g)\g'_+$,
by all $\prt_{lr}$ such that $\ve_l-\ve_r\in \Pi^+_{\g'}$. Observe that $z_{i-1}$ can be arranged into
a linear combination of chains, which are killed by  $\prt_{lr}$, as in Lemma \ref{chains_killed}.
\end{proof}
The elements  $z_{i}$, $i=1,\ldots,N-1$,
belong to the normalizer $\Nc$. They form the set of positive generators of $Z_q(\g,\g')$
for symplectic $\g$. In the orthogonal case, the positive part of $Z_q(\g,\g')$ is generated by $z_{i}$, $i=1,\ldots,N-2$.

\vspace{20pt}
{\bf Acknowledgements}.
This research is supported in part by the RFBR grant 12-01-00207-a.
We are grateful to the Max-Plank Institute for Mathematics in Bonn for hospitality and
excellent research atmosphere.

\end{document}